\documentclass[]{interact}
\usepackage{latexsym}
\usepackage{anyfontsize} 
\usepackage{tikz} 
\usetikzlibrary{cd} 
\usepackage{color}
\usepackage[table]{xcolor}
\usepackage{hyperref}
\usepackage{url}
\usepackage{breakurl}
\usepackage{graphicx}
\usepackage{multirow}
\usepackage{booktabs}
\usepackage{subcaption,float}
\newcommand{\bburl}[1]{\textcolor{blue}{\url{#1}}}

\makeatletter
\newcommand{\monthyear}[1]{%
  \def\@monthyear{\uppercase{#1}}}
\newcommand{\volnumber}[1]{%
  \def\@volnumber{\uppercase{#1}}}
\makeatother

\theoremstyle{plain}
\numberwithin{equation}{section} 
\newtheorem{thm}{Theorem}[section]
\newtheorem{theorem}[thm]{Theorem}
\newtheorem{lemma}[thm]{Lemma}

\numberwithin{table}{section} 
\numberwithin{figure}{section}

\begin{document}

\monthyear{Month Year}
\volnumber{Volume, Number}
\setcounter{page}{1}

\title{Average Degree of Graphs Derived From The Ammann A2 Aperiodic Tiling}

\author{
\name{Darren C. Ong\textsuperscript{a,b} and Xinyan Xu\textsuperscript{a}\thanks{CONTACT Xinyan Xu Email: xinyan.x\_84@qq.com, Darren C. Ong Email: darrenong@xmu.edu.my }}
\affil{\textsuperscript{a}Department of Mathematics, 
                        Xiamen University Malaysia, 
                        Jalan Sunsuria, Bandar Sunsuria, 
                        43900 Sepang, Selangor, Malaysia;
        \textsuperscript{b}School of Mathematical Sciences, 
                        Xiamen University, Xiamen 361005, 
                        Fujian, China }
}

\maketitle

\begin{abstract}
The Ammann A2 tiling is a simple aperiodically ordered tiling of the plane. 
We consider the graph derived from this tiling, by treating each corner of each tile as a vertex and each side of each tile as an edge. We  present a closed-form formula for the average degree of the graph corresponding to this Ammann A2 tiling.
\end{abstract}


\section{Introduction}

A plane tiling is a countable family of closed sets which cover the plane without gaps or overlaps\cite{grunbaum1987tilings}.

We say that a square tiling of $\mathbb R^2$ is periodic, since it is invariant under a nontrivial translation: translating every tile up, down, left or right by a distance equal to the square length will result in the exact same tiling. An aperiodic tiling is a tiling of the plane that is not translation invariant in this way. The most famous example is probably the Penrose tilings, in particular the Kite-Dart and the Rhombus Penrose tilings which each tile the plane aperiodically using only two tile shapes \cite{1974The} \cite{penrose1979pentaplexity}.

In 1992, Robert Ammann, Branko Gr{\"u}nbaum and Geoffrey C Shephard identified a total of four aperiodic tilings: A2, A3, A4, A5. In this paper, we discuss only the Ammann A2 tiling (also known as the Ammann chair), which consists of two hexagonal tiles\cite{ammann1992aperiodic}.

These aperiodic substitution tilings are important as two-dimensional models of quasicrystals. See \cite{BaakeGrimm2013} and \cite{BaakeGrimm2017} for a more detailed exposition of these tilings from this mathematical crystallography perspective. Also, see the Tilings Encyclopedia website \cite{TilingsEncyclopedia} for a compilation of different types of aperiodic tilings.

In this paper we will consider a graph corresponding to the Ammann A2 tiling. That is, given a tiling of a plane or a subset of a plane, the graph corresponding to it arises when we treat each corner of every tile as a vertex of the graph, and each side of the tile as an edge.

These graphs corresponding to the Penrose tilings and other aperiodic subtitution tilings have proved useful in physics and engineering: see \cite{flicker2020classical}, \cite{lloyd2022statistical},  \cite{ma2022quasiperiodic}, \cite{didari2023penrose}, and \cite{koga2020superlattice}. There have also been mathematical explorations of these types of graphs, for example \cite{de2017random} which studies random walks on a graph corresponding to an aperiodic tiling,  and  \cite{SinghLloydFlicker2024} which studies Hamiltonian cycles on graphs corresponding to Ammann-Beenker tilings.

Nevertheless, very little is known about these graphs. This paper aims to fill this gap in the literature by exploring methods to calculate their average degree. We will introduce an explicit formula for calculating the average degree of the graph corresponding to the A2 tiling. 

There has been some work about analytically calculating vertex frqeuencies of aperiodic tilings (\cite{henley},\cite{jaric}, \cite{peyriere} which can be straightforwardly applied to calculating average degree of the corresponding graphs. But to our knowledge, ours is the first formula for the average degree of the graph corresponding to the A2 tiling specifically.

~\\
\textbf{Acknowledgements.} This paper is an expanded version of a thesis by one of the authors \cite{XuXinyan2024}. D.~C.~O.\ was supported in part by a grant from the Fundamental Research Grant Scheme from the Malaysian Ministry of Education (grant number FRGS/1/2022/TK07/XMU/01/1), a grant from the National Natural Science Foundation of China  (grant number 12201524), and a Xiamen University Malaysia Research Fund (grant number XMUMRF/2023C11/IMAT/0024). We thank one of the anonymous reviewers for many helpful comments.
\section{Aperiodic tilings of the plane}

\subsection{Ammann A2}
The Ammann A2 tiling was first discovered by Robert Ammann in 1977.  \cite{grunbaum1987tilings}. This tiling consists of two right-angled irregular hexagon tiles as shown in Figure \ref{A2tiles}. The two tiles are of the same shape but different sizes. All the angles of the hexagons are right angles. Let $\phi=(1+\sqrt{5})/2$ be the golden ratio, and let $\Psi=1/\sqrt{\phi}$ be the reciprocal of the square root of the golden ratio. If we set the longest side of the big hexagon as length $1$, the sides of the big hexagon (listed clockwise starting with the longest side) are of length $1, \Psi, \Psi^2, \Psi^5,\Psi^4, \Psi^3$. The small hexagon is a $\Psi$-scaled version of the large hexagon. 

Notice that, since $\Psi$ is the reciprocal of the square root of the golden ratio
\begin{equation}
    {\Psi}^4 + {\Psi}^2 = 1
\end{equation}
Which means,
\begin{equation}
    {\Psi}^{n} / {\Psi}^{n+2} = \dfrac{\sqrt{5}+1}{2} \text{ and } {\Psi}^{n+4} + {\Psi}^{n+2} = {\Psi}^n.
\end{equation}

\begin{figure}[H]
\centering
\begin{subfigure}{0.48\linewidth}
  \centering
  \includegraphics[scale=0.1]{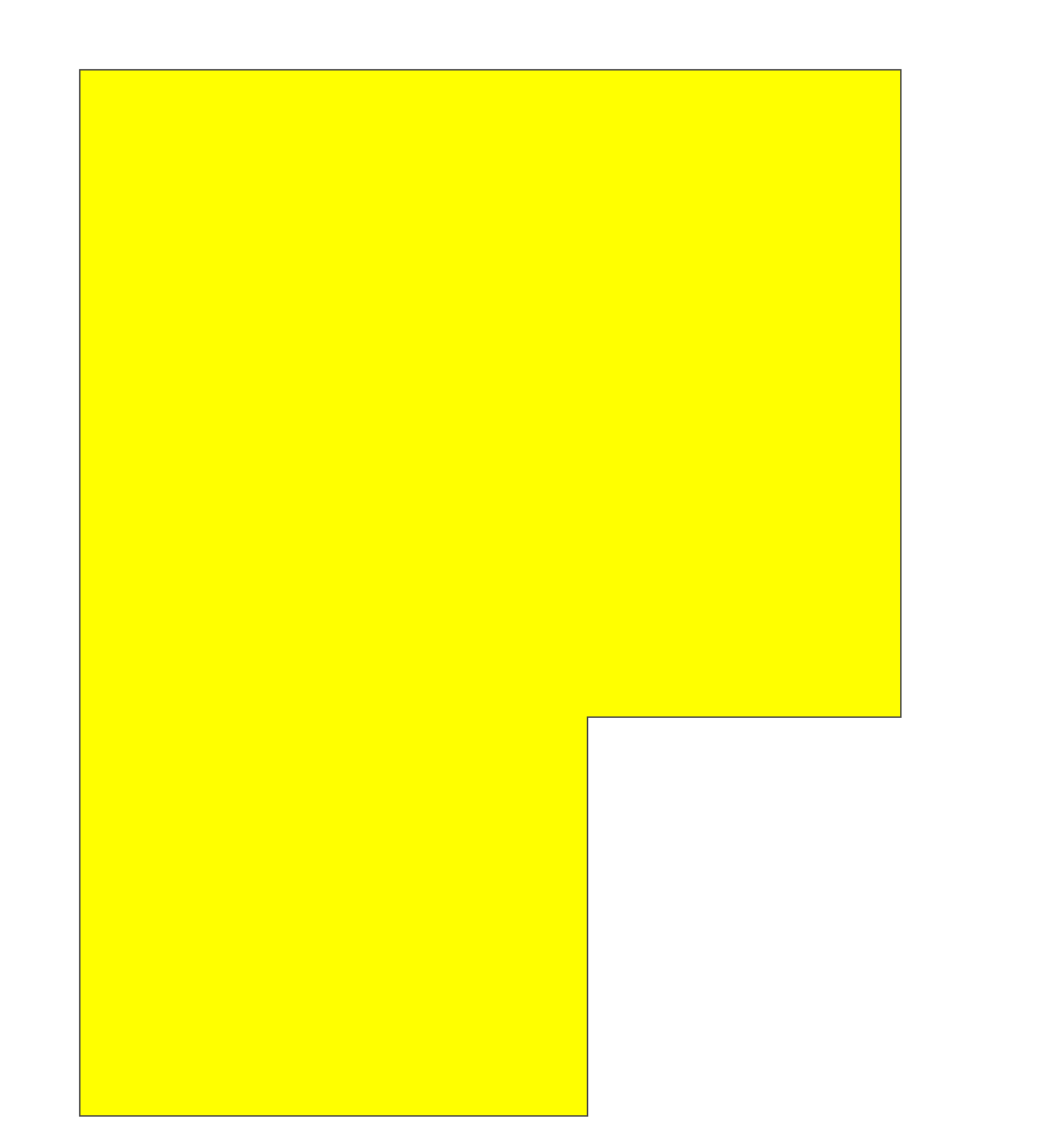}
  \caption{Small Hexagon}    
\end{subfigure}
\hfill
\begin{subfigure}{0.48\linewidth}
  \centering
  \includegraphics[scale=0.27]{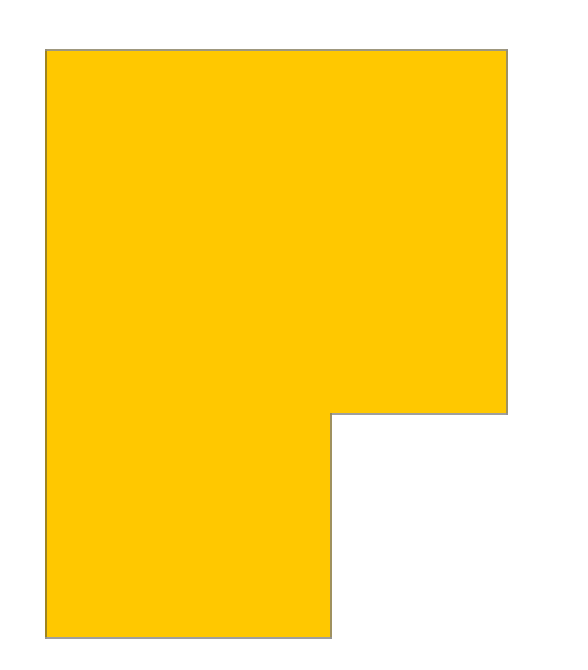}
  \caption{Big Hexagon}    
\end{subfigure}
\caption{Ammann A2 tiles}
\label{A2tiles}
\end{figure}

The substitution rule for the Ammann A2 tiling proceeds as follows. We start with a small hexagon tile. Then, in each step of the substitution process,  every old small hexagon tile gets replaced by a new large hexagon tile, and every old large hexagon tile gets replaced by a new large hexagon tile and a new hexagon tile in an arrangement shown in  Figure \ref{Golden Bee} below. The new substituted tiles are scaled by $\Psi$ compared to the old original tiles. This substitution rule is visualized in Figure \ref{subrule}.  All the lengths of the sides of the tiles in every substitution step can be written as a constant times the power of $\Psi$ \cite{durand2020structure} (in this paper we choose the constant to be $1$) .

\begin{figure}[H]
    \centering
    \includegraphics[scale = 0.13]{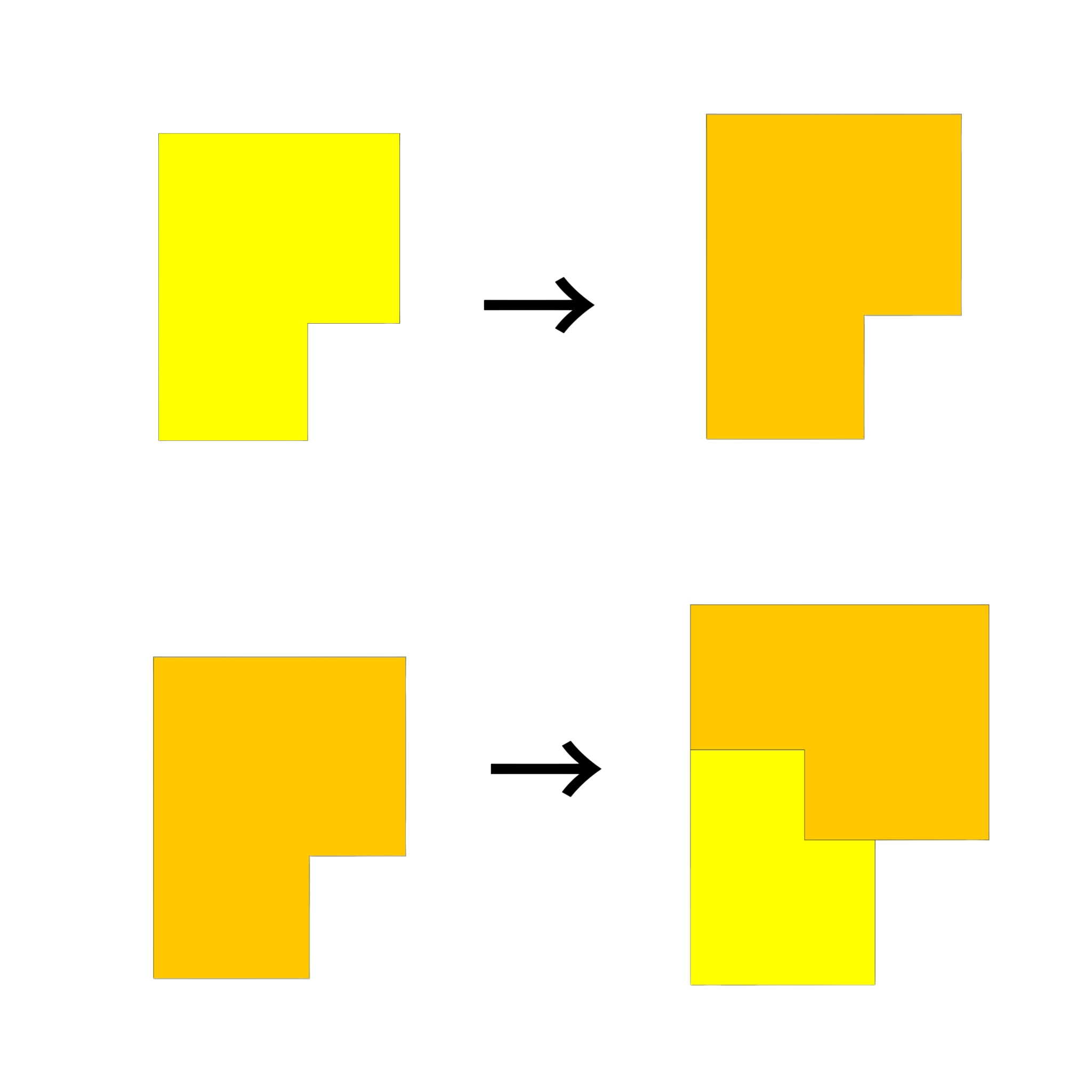}
    \caption{The substitution rule for the Ammann A2 tiling}
    \label{subrule}
\end{figure}

\begin{figure}[H]
    \centering
    \includegraphics[scale = 0.23]{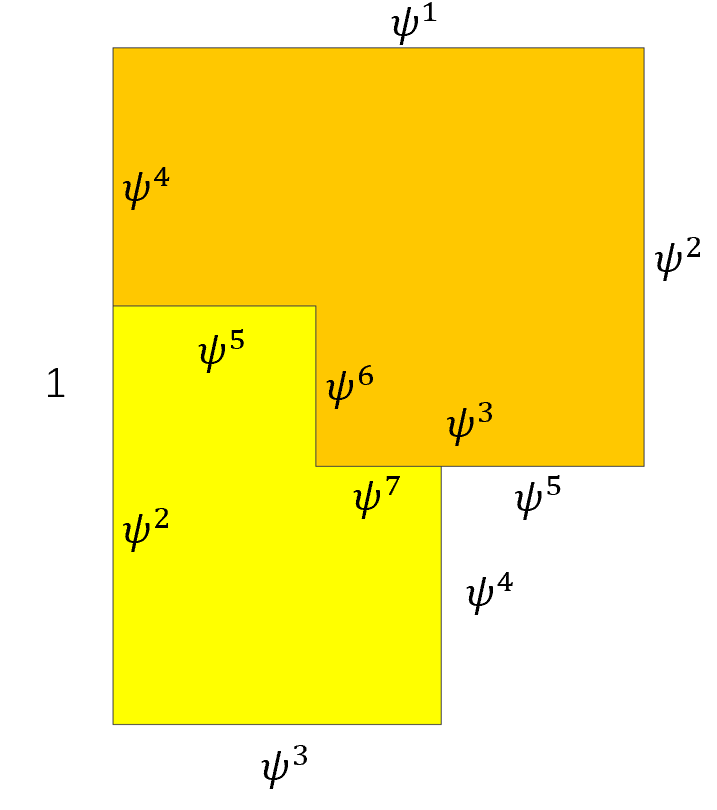}
    \caption{A single large hexagon tile gets replaced by a large hexagon tile and a small hexagon tile, scaled by $\Psi$}
    \label{Golden Bee}
\end{figure}

Let us call the small hexagon tile "generation 1", and let us call the tile arrangement we obtain after applying the substitution rule $n$ times to a single small hexagon tile "generation $n+1$"  of the Ammann A2 tiling. Then we notice that each new generation of the Ammann A2 tiling is composed of the previous two generations, mimicking the Fibonacci recursion.For instance, in Figure \ref{345} the following three generations are illustrated: generation 3, generation 4 and generation 5. Specifically, generation 5 is composed of generation 4 rotated 90 degrees clockwise and generation 3 flipped vertically.

\begin{figure}[H]
 \centering
  \begin{subfigure}{0.32\textwidth}
 \centering
 \includegraphics[scale=0.15]{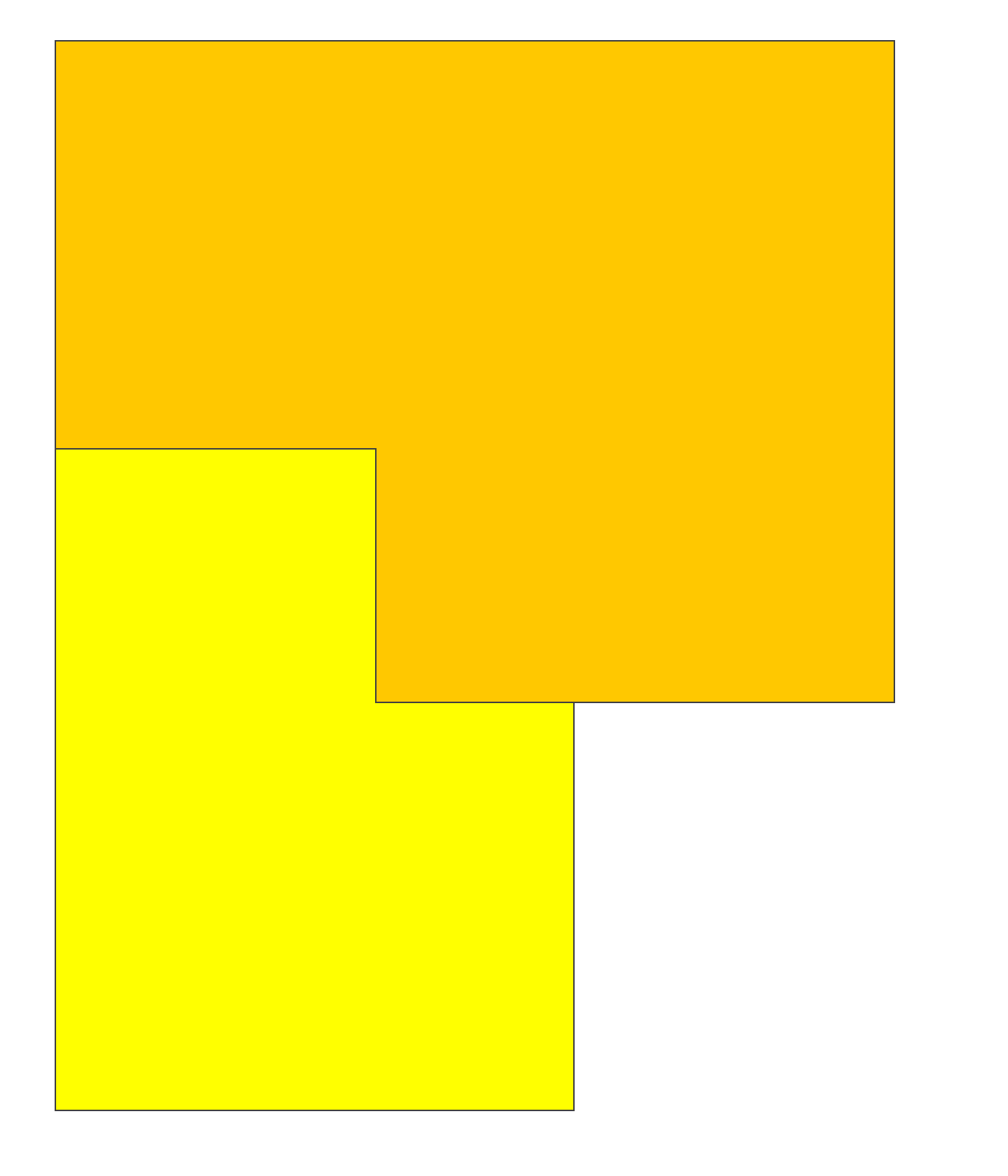}
 \caption{Generation 3}    
 \end{subfigure}
 \begin{subfigure}{0.32\textwidth}
 \centering
 \includegraphics[scale=0.15]{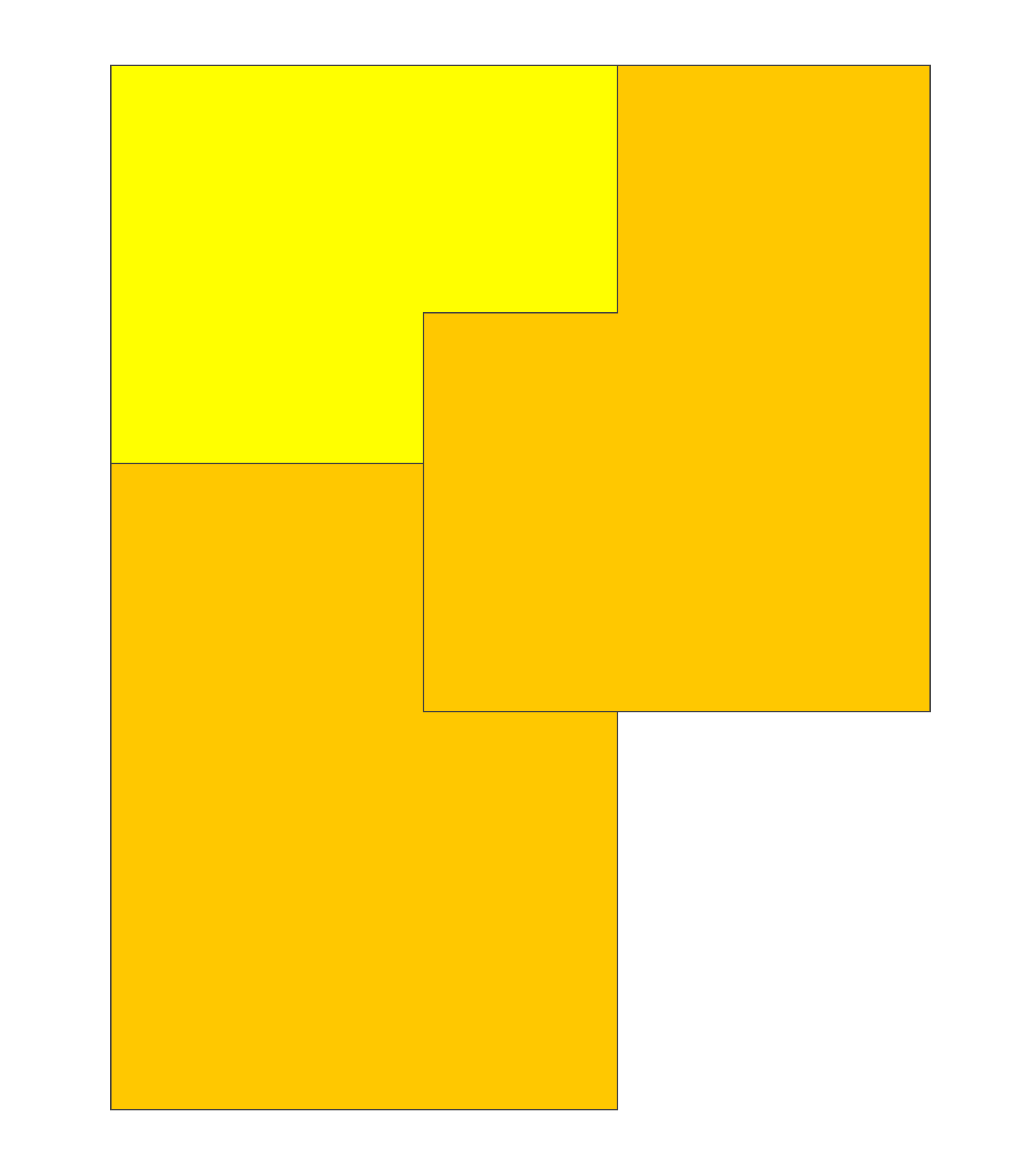}
 \caption{Generation 4}    
 \end{subfigure}
\centering
 \begin{subfigure}{0.32\textwidth}
 \centering
 \includegraphics[scale=0.15]{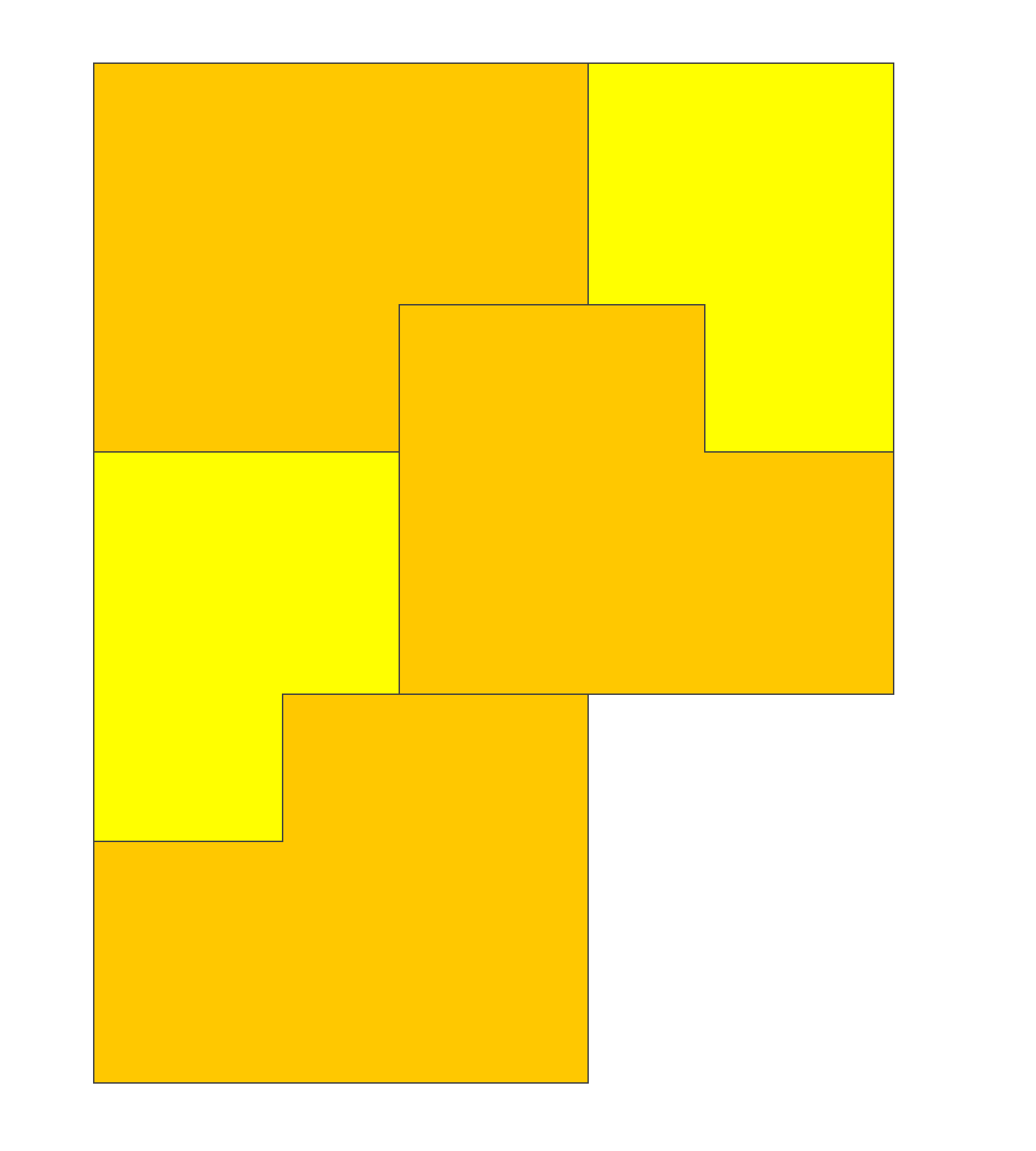}
 \caption{Generation 5}    
 \end{subfigure}
\caption{Generation 5 is a combination of the generation 4 and generation 3 tiles.}
\label{345}
\end{figure}

For every $n$, it is easy to show that Generation $n+2$ is composed by combining the tiles from Generation $n+1$ and Generation $n$ in a similar way.

\begin{figure}[H]
    \centering
    \includegraphics[scale = 0.4]{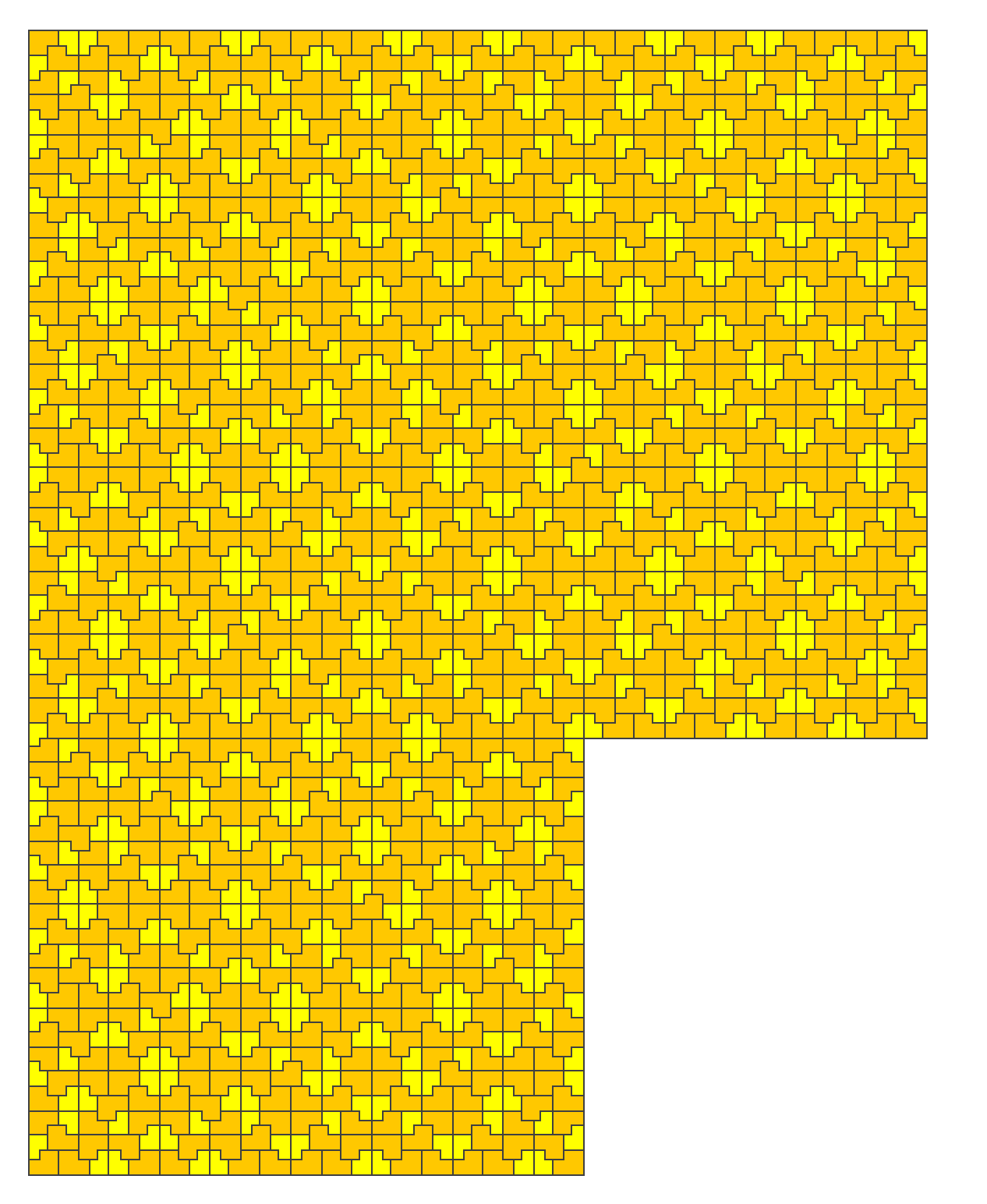}
    \caption{The Ammann A2 tiling after several generations}
\end{figure}

\section{Average degree formula for A2 tiling graph}
Given a tiling of a subset of the plane, we may obtain a graph by treating every corner of a tile as a vertex, and every straight line boundary of a tile as an edge. We will use the term ``point" interchangably with ``vertex".
 The degree of a point is equal to the number of edges connecting to the point. Each edge connects two vertices, so each edge adds contributes two degrees to the sum of all the degrees of all the vertices. The average degree of a graph can then be calculated as follows. \cite{west2001introduction}
\begin{equation}
    \begin{split}
        \text{Average degree of a graph} 
            &= \frac{\text{Sum of Degrees of all Vertices}}{\text{Number of Vertices}} \\[2ex]
            &= \frac{2 \times \text{ Number of Edges}}{\text{ Number of Vertices}}
    \end{split}
\end{equation}
We will provide a closed form formula for the average degree of every generation of the Ammann A2 tiling. 

\subsection{Recursion Formula}
For the Ammann A2 tiling, we can determine the average degree of the limiting graph explicitly.
Let us first define F$\left(n\right)$ as the $n$th Fibonacci number (indexed so $F(1)=1, F(2)=1, F(3)=2$). We will henceforth refer to generation $n$ of the A2 tiling as the A2-$n$ tiling. Refer to Figures \ref{A2-1}, \ref{A2-2}, and \ref{A2-3} for pictures of A2-1, A2-2 and A2-3. We then have the following theorem for the average degree of the graph corresponding to the A2-$k$ tiling.

\begin{theorem}\label{A2theorem}
Consider the graph corresponding to the A2-$k$ tiling by treating every corner of every tile as a vertex, and every side of every tile as an edge. Let $V(k)$ be the number of vertices in that graph, and let $T(k)$ be the sum of the degrees of all the vertices. Let $F(n)$ be the $n$th Fibonacci number. Then for $k\geq 7$,

\begin{equation}\label{Vformula}
V(k)=\begin{cases}
1+\frac{F\left(\frac{k}{2}\right)}{2}+F\left(\frac{k}{2}+1\right) +2F\left(\frac{k}{2}-1\right)-\frac{F\left(\frac{k}{2}-3\right)}{2}\\+\left(\frac{31}{2}+\frac{69}{2\sqrt 5}\right)\left(\frac{1+\sqrt 5}{2}\right)^{k-7}+\left(\frac{31}{2}-\frac{69}{2\sqrt 5}\right)\left(\frac{1-\sqrt 5}{2}\right)^{k-7}, & \text{ $k$ even}\\
1+2F\left(\frac{k+1}{2}\right)+F\left(\frac{k-1}{2}\right)\\+\left(\frac{31}{2}+\frac{69}{2\sqrt 5}\right)\left(\frac{1+\sqrt 5}{2}\right)^{k-7}+\left(\frac{31}{2}-\frac{69}{2\sqrt 5}\right)\left(\frac{1-\sqrt 5}{2}\right)^{k-7}, & \text{ $k$ odd}
\end{cases}
\end{equation}

and

\begin{equation}\label{Tformula}
T(k)=\begin{cases}
F\left(\frac{k}{2}\right)+2F\left(\frac{k}{2}+1\right) +4F\left(\frac{k}{2}-1\right)-F\left(\frac{k}{2}-3\right)\\+\left(44+\frac{98}{\sqrt 5}\right)\left(\frac{1+\sqrt 5}{2}\right)^{k-7}+\left(44-\frac{98}{\sqrt 5}\right)\left(\frac{1-\sqrt 5}{2}\right)^{k-7}, & \text{ $k$ even}\\
4F\left(\frac{k+1}{2}\right)+2F\left(\frac{k-1}{2}\right)\\+\left(44+\frac{98}{\sqrt 5}\right)\left(\frac{1+\sqrt 5}{2}\right)^{k-7}+\left(44-\frac{98}{\sqrt 5}\right)\left(\frac{1-\sqrt 5}{2}\right)^{k-7}, & \text{ $k$ odd}
\end{cases}
\end{equation}
The average degree of the graph corresponding to A2-$k$ is then given by $T(k)/V(k)$, and the average degree of the limiting graph is given by

\begin{equation}
\lim_{k\to\infty} \frac{T(k)}{V(k)}= \frac{\left(44+\frac{98}{\sqrt 5}\right)}{\left(\frac{31}{2}+\frac{69}{2\sqrt 5}\right)}=\frac{2}{11}\left(14+\phi\right)\approx 2.839642543409072\ldots.
\end{equation}
    
\end{theorem}

For the sake of completeness, we list $V(k)$ and $T(k)$ for $k=1,2,\ldots,6$ as well. These graphs are small enough that $V(k)$ and $T(k)$ can be counted by hand.

\begin{table}[H]
    \centering
    \begin{tabular}{|c|c|c|}
       \hline
        Generation, A2-$k$ & Total Number of Vertices, $V(k)$ & 2 $\times$ \text{Total Number of Edges, $T(k)$} \\
           \hline
           A2-1  & 6 & 2 $\times$ 6 \\
           \hline
           A2-2  & 6 & 2 $\times$ 6 \\
           \hline
           A2-3  & 9 & 2 $\times$ 10 \\
           \hline
           A2-4  & 12 & 2 $\times$ 14 \\
           \hline
           A2-5  & 18 & 2 $\times$ 22 \\
           \hline
           A2-6  & 26 & 2 $\times$ 33 \\
           \hline
    \end{tabular}
    \caption{Data for A2-1 through A2-6}
    \label{tab:my_label}
    
\end{table}
Before we proceed with the proof of the theorem, let us make some observations about the A2-$k$ graphs.

As previously mentioned, the Ammann A2 tiling follows this rule: the generation $n$ tiling is obtained by combining the $\left(n-1\right)$ th generation and the $\left(n-2\right)$th generation. Therefore, the total number of vertices and the sum of the degrees of all vertices of the generation $n$ graph are obtained by taking a sum from the graphs corresponding to the previous two generations, adjusted for the changes at the intersection line where the graphs of the previous two generations combine.

If we are focus on the total number of vertices and the sum of the degrees of all vertices for the previous two generations, as well as the changes in the degrees and the number of vertices when combining the previous two generations, we can also calculate the sum of the degrees of all vertices and total number of vertices for the $n$th generation. Therefore, our task is to identify the changes that occur when combining the previous two graphs. The only change occurs at the common boundary of the two previous generations. 

 Let us begin with the A2-1 and A2-2 tilings, each consisting of 6 points and characterized by a total of 12 degrees. To obtain A2-3, we combine the A2-1 and A2-2 tilings. There are four vertices in the intersection line of the A2-1 and A2-2 tilings. We then compare the number of vertices and degrees at those four positions after the two parts are combined, compared to before the two parts are combined. 

\begin{figure}[H]
\centering
 \begin{subfigure}{0.3\textwidth}
 \centering
 \includegraphics[scale=0.18]{S.png}
 \caption{A2-1}  
 \label{A2-1}
 \end{subfigure}
\centering
 \begin{subfigure}{0.3\textwidth}
 \centering
 \includegraphics[scale=0.45]{B.png}
 \caption{A2-2}    
 \label{A2-2}
 \end{subfigure}
 \centering
 \begin{subfigure}{0.3\textwidth}
 \centering
 \includegraphics[scale=0.48]{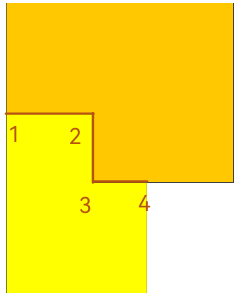}
 \caption{A2-3}   
 \label{A2-3}
 \end{subfigure}
 \end{figure}

 \begin{table}[H]
    \centering
    \begin{tabular}{|c|c|c|c|c|}
    \hline
        \multirow{2}{*}{Generation} & \multicolumn{2}{c|}{Decrease in} & \multirow{2}{*}{Point No.} & \multirow{2}{*}{Amount} \\
        \cline{2-3}
           & Vertices & Degree &   &   \\
        \hline
         \multirow{3}{*}{A2-3} & 1 & 1 & 1 & 1 \\
        \cline{2-5}
         & 1 & 2 & 2 , 3 & 2 \\
         \cline{2-5}
          & 0 & -1 & 4 & 1 \\
        \hline
    \end{tabular}
\end{table}
A2-3 is formed by combining A2-2 tiling rotated 90 degrees and A2-1 tiling flipped vertically. The red intersection line in A2-3 illustrates the changes at each point positions relative to the separated states of A2-1 and A2-2:
\begin{itemize}
    \item At the No.1 point position, there is a decrease of 1 point and 1 degree compared to when the two parts are separated.
    \item At the No.2 and No.3 positions, there is a decrease of 1 point and 2 degrees of each.
    \item At the No.4 position, there is no point reduction because the A2-2 tiling does not contribute a vertex at this position, hence it starts at 0 degrees. The vertex in the A2-1 tiling at the No.4 position, which initially has 2 degrees, increases to 3 degrees after the connection, resulting in a 1 degree increase at this position.
\end{itemize}

While our table indicates the number of points and degree decreases, it uses -1 to signify an actual increase of 1 degree, as exemplified in the No.4 position.

\begin{figure}[H]
\centering
 \begin{subfigure}{0.3\textwidth}
 \centering
 \includegraphics[scale=0.45]{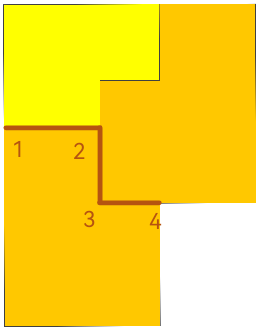}
 \caption{A2-4}  
 \label{A2-4}
 \end{subfigure}
\centering
 \begin{subfigure}{0.3\textwidth}
 \centering
 \includegraphics[scale=0.45]{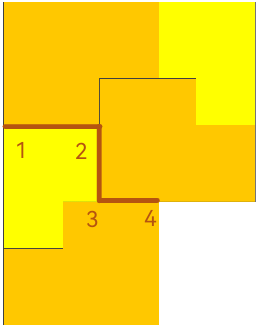}
 \caption{A2-5}    
 \label{A2-5}
 \end{subfigure}
 \centering
 \begin{subfigure}{0.3\textwidth}
 \centering
 \includegraphics[scale=0.43]{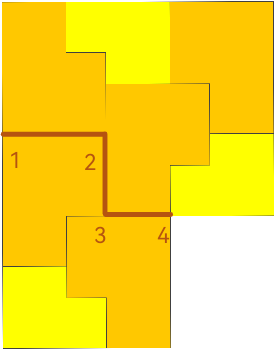}
 \caption{A2-6}   
 \label{A2-6}
 \end{subfigure}
\end{figure}

 \begin{table}[H]
    \centering
    \begin{tabular}{|c|c|c|c|c|}
    \hline
        \multirow{2}{*}{Generation} & \multicolumn{2}{c|}{Decrease in} & \multirow{2}{*}{Point No.} & \multirow{2}{*}{Amount} \\
        \cline{2-3}
           & Point & Degree &   &   \\
        \hline
        
         \multirow{3}{*}{A2-4} & 1 & 1 & 1 & 1 \\
        \cline{2-5}
         & 1 & 2 & 2 , 3 & 2 \\
         \cline{2-5}
          & 0 & -1 & 4 & 1 \\
        \hline

        \multirow{3}{*}{A2-5} & 1 & 1 & 1 & 1 \\
        \cline{2-5}
         & 1 & 2 & 2 , 3 & 2 \\
         \cline{2-5}
          & 0 & -1 & 4 & 1 \\
        \hline

        \multirow{2}{*}{A2-6} & 1 & 1 & 1 , 4 & 2 \\
        \cline{2-5}
         & 1 & 2 & 2 , 3 & 2 \\
         \hline
    \end{tabular}
\end{table}
Denote the change in the intersection line in a new way: 
\begin{center}
    $\left(change\: in \:point\: number\:,\:change\: in\: degree\: number\right)$ $\times$ $amount$
\end{center}
So the change in A2-4, could be denoted as:

$\left(1\:,\:1\right)\times1 \qquad\& \qquad \left(1\:,\:2\right)\times2 \qquad \& \qquad \left(0\:,\:-1\right)\times1$ 

Change in A2-5:

$\left(1\:,\:1\right)\times1 \qquad\& \qquad \left(1\:,\:2\right)\times2 \qquad \& \qquad \left(0\:,\:-1\right)\times1$ 

Change in A2-6:

$\left(1\:,\:1\right)\times2 \qquad\& \qquad \left(1\:,\:2\right)\times2$ 
\begin{figure}[H]
\begin{subfigure}{0.3\textwidth}
 \centering
 \includegraphics[scale=0.45]{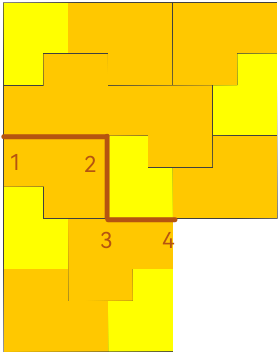}
 \caption{A2-7}  
 \label{A2-7}
 \end{subfigure}
\centering
 \begin{subfigure}{0.3\textwidth}
 \centering
 \includegraphics[scale=0.475]{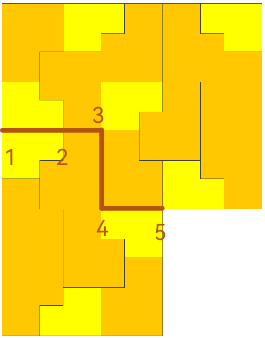}
 \caption{A2-8}    
 \label{A2-8}
 \end{subfigure}
\centering
 \begin{subfigure}{0.3\textwidth}
 \centering
 \includegraphics[scale=0.375]{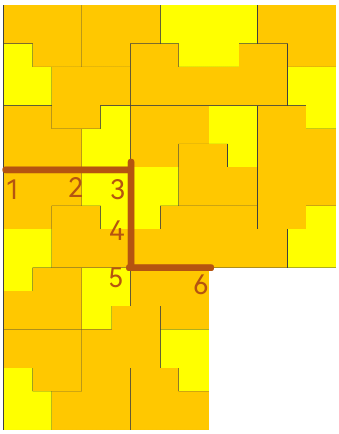}
 \caption{A2-9}    
 \label{A2-9}
 \end{subfigure}
\end{figure}

\begin{table}[H]
    \centering
    \begin{tabular}{|c|c|c|c|c|}
    \hline
        \multirow{2}{*}{Generation} & \multicolumn{2}{c|}{Decrease in} & \multirow{2}{*}{Point No.} & \multirow{2}{*}{Amount} \\
        \cline{2-3}
           & Point & Degree &   &   \\
        \hline
        \multirow{2}{*}{A2-7} & 1 & 1 & 1 , 4 & 2 \\
        \cline{2-5}
         & 1 & 2 & 2 , 3 & \cellcolor{gray!30}2 \\
         \hline

         \multirow{2}{*}{A2-8} & 1 & 1 & 1 , 5 & 2 \\
        \cline{2-5}
         & 1 & 2 & 2 , 3 , 4 & \cellcolor{gray!30}3 \\
         \hline
         
         \multirow{2}{*}{A2-9} & 1 & 1 & 1 , 6 & 2 \\
        \cline{2-5}
         & 1 & 2 & 2 - 5 & \cellcolor{gray!30}4 \\
         \hline
         \end{tabular}
         \end{table}

Change in A2-7:
    
    $\left(1\:,\:1\right)\times2 \qquad\& \qquad \left(1\:,\:2\right)\times2$ 

Change in A2-8:

$\left(1\:,\:1\right)\times2 \qquad\& \qquad \left(1\:,\:2\right)\times3$ 

Change in A2-9:

$\left(1\:,\:1\right)\times2 \qquad\& \qquad \left(1\:,\:2\right)\times4$

A fundamental rule applies here: 

The points at the two ends always experience a decrease of 1 point and a loss of 1 degree, while the points in the middle section always decrease by 1 point and lose 2 degrees. 

Concurrently, the number of middle points from A2-7 to A2-9 is gradually increasing: there are 2 middle points in A2-7, 3 in A2-8, and 4 in A2-9.

\begin{figure}[H]
\centering
 \begin{subfigure}{0.32\textwidth}
 \centering
 \includegraphics[scale=0.23]{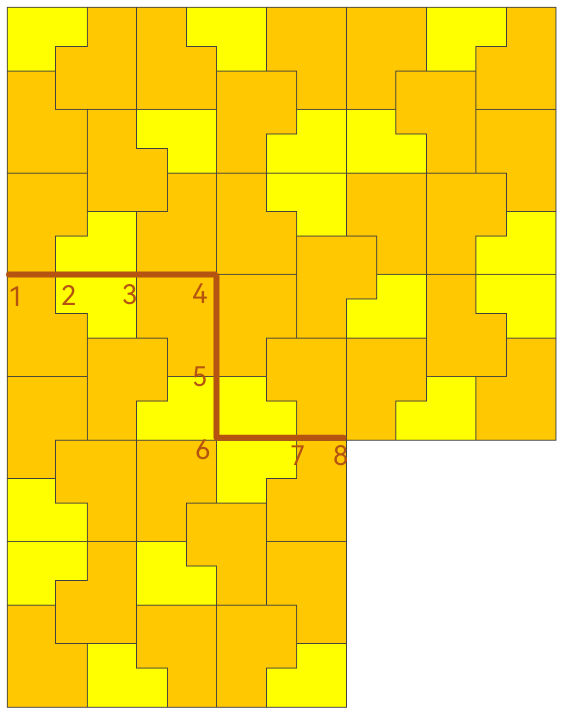}
 \caption{A2-10}  
 \label{A2-10}
 \end{subfigure}
\centering
 \begin{subfigure}{0.32\textwidth}
 \centering
 \includegraphics[scale=0.235]{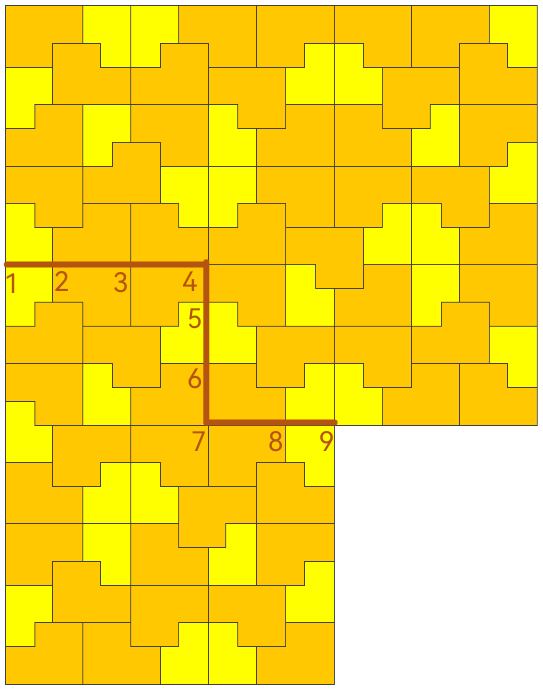}
 \caption{A2-11}    
 \label{A2-11}
 \end{subfigure}
\centering
 \begin{subfigure}{0.25\textwidth}
 \centering
 \includegraphics[scale=0.235]{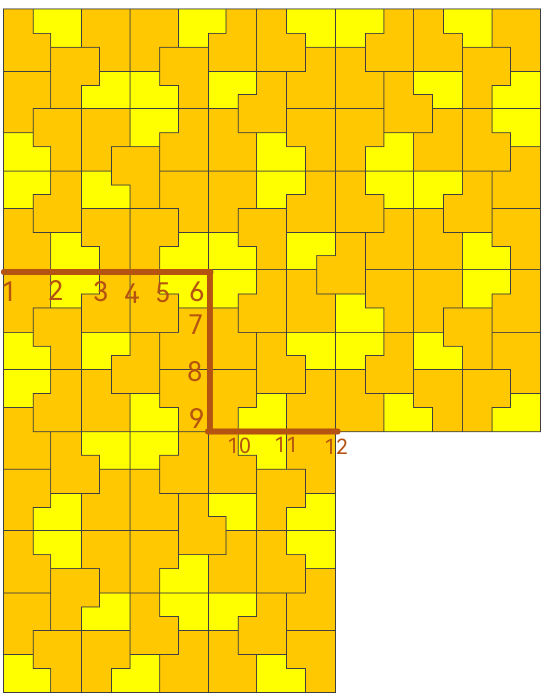}
 \caption{A2-12}    
 \label{A2-12}
 \end{subfigure}
\end{figure}

\begin{table}[H]
    \centering
    \begin{tabular}{|c|c|c|c|c|}
    \hline
        \multirow{2}{*}{Generation} & \multicolumn{2}{c|}{Decrease in} & \multirow{2}{*}{Point No.} & \multirow{2}{*}{Amount} \\
        \cline{2-3}
           & Point & Degree &   &   \\
        \hline

        \multirow{2}{*}{A2-10} & 1 & 1 & 1 , 8 & 2 \\
        \cline{2-5}
         & 1 & 2 & 2 - 7 & \cellcolor{gray!30}6 \\
         \hline

         \multirow{2}{*}{A2-11} & 1 & 1 & 1 , 9 & 2 \\
        \cline{2-5}
         & 1 & 2 & 2 - 8 & \cellcolor{gray!30}7 \\
         \hline

         \multirow{2}{*}{A2-12} & 1 & 1 & 1 , 12 & 2 \\
        \cline{2-5}
         & 1 & 2 & 2 - 11 & \cellcolor{gray!30}10 \\
         \hline

    \end{tabular}
\end{table}

Change in A2-10:
    
    $\left(1\:,\:1\right)\times2 \qquad\& \qquad \left(1\:,\:2\right)\times6$ 

Change in A2-11:

$\left(1\:,\:1\right)\times2 \qquad\& \qquad \left(1\:,\:2\right)\times7$ 

Change in A2-12:

$\left(1\:,\:1\right)\times2 \qquad\& \qquad \left(1\:,\:2\right)\times10$ 

The change of the two ends are still $\left(1\:,\:1\right)$ and in the middle points are also $\left(1\:,\:2\right)$. The amount of middle points from A2-10 to A2-12 are 6, 7, 10.

Starting from A2-7, there are only two types of changes occurring along the intersection line:
\begin{itemize}
    \item The changes at the two ends: At each end, there is a reduction of 1 point and 1 degree, and this pattern occurs twice $\left(once\: at \:each\: end\right)$.
    \item The changes in the middle part: There is a decrease of 1 point and 2 degrees. The number of such occurrences varies with each generation.
\end{itemize}

To discern the underlying rule, consider the number of middle points of A2-7 as the first term in the following sequence.
\begin{table}[h]
    \resizebox{\textwidth}{!}{
    \begin{tabular}{|c|c|c|c|c|c|c|c|c|c|c|}
    \hline
    A2-7 & A2-8 & A2-9 & A2-10 & A2-11 & A2-12 & A2-13 & A2-14 & A2-15 & A2-16 & A2-17 \\
    \hline
        $D_1$ & $D_2$ & $D_3$ & $D_4$ & $D_5$ & $D_6$ & $D_7$ & $D_8$ & $D_9$ & $D_{10}$ & $D_{11}$  \\
        \hline
        2 & 3 & 4 & 6 & 7 & 10 & 12 & 17 & 20 & 28 & 33 \\
        \hline
    \end{tabular}}
    \caption{Relation between generation and amount of middle points}
\end{table}

We calculate the difference in the number of middle points between consecutive generations, expressed as $\left( D_{N+1} - D_N \right)$, where $D_N$ represents the number of middle points in A2-$(N+6)$.
\begin{table}[h]
    \centering
    \begin{tabular}{|c|c|c|c|c|c|c|c|c|c|}
    \hline
        1 & 1 & 2 & 1 & 3 & 2 & 5 & 3 & 8 & 5 \\
        \hline
    \end{tabular}
    \caption{Differences of middle point amount between two generations}
    \label{tab:my_label2}
\end{table}
Divide this sequence of differences into two sequences $\{a_n\}$ and $\{b_n\}$ with $a_n=D_{2n}-D_{2n-1}$ and $b_n=D_{2n+1}-D_{2n}$.
\begin{table}[h]
    \centering
    \begin{tabular}{|c|c|c|c|c|c|c|c|c|c|}
    \hline
      \cellcolor{blue!20}$a_1$ & \cellcolor{purple!20}$b_1$ & \cellcolor{blue!20}$a_2$ & \cellcolor{purple!20}$b_2$ & 
      \cellcolor{blue!20}$a_3$ & \cellcolor{purple!20}$b_3$ & 
      \cellcolor{blue!20}$a_4$ & \cellcolor{purple!20}$b_4$ & 
      \cellcolor{blue!20}$a_5$ & \cellcolor{purple!20}$b_5$  \\
      \hline
      \cellcolor{blue!20}1 & \cellcolor{purple!20}1 & 
      \cellcolor{blue!20}2 & \cellcolor{purple!20}1 & 
      \cellcolor{blue!20}3 & \cellcolor{purple!20}2 & 
      \cellcolor{blue!20}5 & \cellcolor{purple!20}3 & 
      \cellcolor{blue!20}8 & \cellcolor{purple!20}5 \\
      \hline
    \end{tabular}
    \caption{The first few terms of the $\{a_n\}$ and $\{b_n\}$ sequences}
\end{table}

We observe that $\{a_n\}$ and $\{b_n\}$ are both the Fibonacci sequence. $\{a_n\}$ is the Fibonacci sequence from the 2nd term and $\{b_n\}$ is the Fibonacci sequence from the 1st term. We will proceed to prove this observation, but we need a lemma first:
\begin{lemma}\label{reflection}
For the A2-$k$ tiling with $k\geq 9$, any tile $T$ with one of its sides on the intersection line has a dual tile $T'$ that is a reflection of $T$ through that side. If $T$ has two sides on the intersection line it has two dual tiles $T'$ and $T''$, which are reflections of $T$ through those two sides respectively. 
\end{lemma}
\textbf{Remark.} Note that $T$ has to have a full side (with positive length) on the intersection line, if $T$ only has a corner point on the intersection line this lemma does not apply.

\begin{proof}
The A2-9 tiling is small enough that we can verify this lemma is true for it by inspection of Figure \ref{A2-9}. But then to show our lemma is true for A2-$k$ for $k\geq 10$, we simply observe that the tile substitution algorithm will create reflection symmetric tiles on both sides of an axis of reflection if the original tiles are reflection symmetric about that axis of reflection.
\end{proof}
We are now ready to prove that the $\{a_k\}$ and $\{b_k\}$ are both Fibonacci sequences.
\begin{lemma}
    For $k\geq 1$, $a_k=F(k+1)$ and $b_k=F(k)$, where $F(k)$ is the $k$th Fibonacci number.
\end{lemma}
\begin{proof}
First, we observe from Figure \ref{A2-3} that in the orange (large) tile, the only side where a new point appears after one substitution step is the longest side. From this observation, in Figure \ref{sidenumbers} we label each side of both the orange and yellow tiles with a number, indicating how many rounds of substitution must occur for a new point to appear on that side. 

\begin{figure}[H]
\centering
\includegraphics[scale = 0.7]{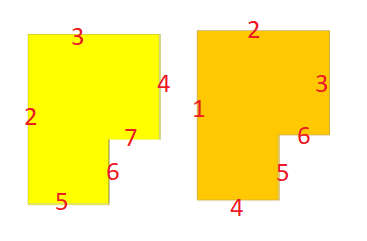}
\caption{The numbers labeled on each sides of the tile represent how many substitution steps have to occur before a new point appears on that side.}
\label{sidenumbers}
\end{figure}

We then observe in Figure \ref{A2-7} that the intersection line there is composed of three sides, which have labels $1,2,3$ according to Figure \ref{sidenumbers}. Similarly, the intersection line in Figure \ref{A2-8}
is composed of four tile sides, with labels $2,4,1,2$. A tile side labeled $1$ gets split into two tile sides, with labels $2$ and $4$ 

By definition, after one substitution step tile sides with labels $2,3,4,5,6$ or $7$ get their labels reduced by $1$. A tile side with label $1$ splits into two tile sides, with labels $4$ and $2$.

Let $L_k$ for $k\geq 7$ represent an unordered list of numbers $1,2,3$ or $4$, which are the labels of the sides on the intersection line of the A2-$k$ tiling. Thus $L_7=(1,2,3)$ and $L_8=(2,4,1,2)$. To obtain $L_{k+1}$, we start with $L_k$ and replace $4\to 3$, replace $3\to 2$, replace $2\to 1$ and replace $1\to 2,4$. Let $|L_k|$ represent the length of the list $L_k$, so $|L_7|=3$ and $|L_8|=4$. It is clear that $D_k=|L_{k+6}|-1$, since the number of vertices on the intersection line is $1$ plus the number of sides of the intersection line, and $D_k$ counts the number of vertices in the intersection line other than the first and the last one. This implies that $a_k=D_{2k}-D_{2k-1}=|L_{2k+6}|-|L_{2k+5}|$, and $b_k=D_{2k+1}-D_{2k}=|L_{2k+7}|-|L_{2k+6}|$.

It is clear that for any $k\geq 7$, $|L_{k+1}|-|L_k|$ is equal to the number of `$1$'s in the list $L_k$ (since this is the only way a new side can be created). Since a side labeled $1$ is replaced by two sides of labels $2$ and $4$ according to the substitution rule, if $k\geq 8$ each `$1$' in $L_k$ arises from a `$2$' in $L_{k-1}$, which (if $k\geq 11$) itself arises from either a `$1$' in $L_{k-2}$  or a `$1$' in $L_{k-4}$.

This line of reasoning implies that if $k\geq 3$,
\begin{align}
a_k=& |L_{2k+6}|-|L_{2k+5}|\nonumber\\
=& (\text{number of `$1$'s in $L_{2k+5}$})\nonumber\\
=& (\text{number of `$1$'s in $L_{2k+3}$})+(\text{number of `$1$'s in $L_{2k+1}$})\nonumber\\
=&a_{k-1}+a_{k-2}.
\end{align}
We can similarly show that for $k\geq 3$, $b_{k}=b_{k-1}+b_{k-2}$. 

In other words, both $\{a_k\}$ and $\{b_k\}$ obey the Fibonacci recursion. The sequence $\{b_k\}$ has initial conditions $b_1=b_2=1$, and $\{a_k\}$ has initial conditions $a_1=1$, $a_2=2$. This concludes our proof.
\end{proof}
We are now able to write a formula for $D_k$ based on the Fibonacci numbers:
\begin{lemma}\label{DkFormula}
For $k\geq 3$, $D_k=-1+F(\lfloor k/2\rfloor+3)+ F(\lfloor (k-1)/2\rfloor+2)$.
\end{lemma}
\begin{proof}
Recall the following formula for the sum of the first $n$ Fibonacci numbers (found in, for instance \cite{hardywright1979})
\begin{equation}
    \sum_{k=1}^n F(k)=F(n+2)-1.
\end{equation}
We know that $D_1=2$, $D_2=3$ and for $k\geq 3$ ,
\begin{align*}
D_k=&D_1+\sum_{j=1}^{\lfloor k/2\rfloor} a_j+\sum_{i=1}^{\lfloor (k-1)/2\rfloor} b_i\\
=& D_1+ \sum_{j=1}^{\lfloor k/2\rfloor} F(j+1)+\sum_{i=1}^{\lfloor (k-1)/2\rfloor} F(i)\\
=& 2+ ((F(\lfloor k/2\rfloor+3)-1-1)+(F(\lfloor (k-1)/2\rfloor+2)-1)\\
=&-1+F(\lfloor k/2\rfloor+3)+ F(\lfloor (k-1)/2\rfloor+2)
\end{align*}
\end{proof}
\begin{proof}[Proof of Theorem \ref{A2theorem}]
The A2-$k$+$2$ tiling is generated by pasting the A2-$k$+$1$ tiling with the A2-$k$ tiling. If $k\geq 7$, by Lemma \ref{reflection} we know that $2+D_{k-4}$ of the points from the A2-$k$+$1$ tiling and the A2-$k$ will overlap on the intersection line of the A2-$k$+$2$ tiling. This implies that for $k\geq 7$

\begin{equation}
V(k+2)=V(k+1)+V(k)-2-D_{k-4}.
\end{equation}
Using Lemma \ref{DkFormula}, we know this is equivalent to
\begin{equation}\label{Vequation}
V(k+2)=V(k+1)+V(k)-1-F(\lfloor (k-4)/2\rfloor+3)- F(\lfloor (k-4-1)/2\rfloor+2).
\end{equation}

This is a second order non-homogeneous difference equation, and we can apply standard methods to find the general solution. It is not hard to verify the following is a particular solution of  \eqref{Vequation}:

\begin{equation}\label{Vparticular}
V_p(k)=\begin{cases}
1+\frac{F\left(\frac{k}{2}\right)}{2}+F\left(\frac{k}{2}+1\right) +2F\left(\frac{k}{2}-1\right)-\frac{F\left(\frac{k}{2}-3\right)}{2}, & \text{ $k$ even}\\
1+2F\left(\frac{k+1}{2}\right)+F\left(\frac{k-1}{2}\right), & \text{ $k$ odd}
\end{cases}
\end{equation}
We now find the complementary solution. This is straightforward, because the homogeneous part of \eqref{Vequation} is just the Fibonacci recursion equation. We thus have the complementary solution

\begin{equation} \label{Vcomplementary} V_c(k)= C_1 \left(\frac{1+\sqrt 5}{2}\right)^{k-7}+ C_2 \left(\frac{1-\sqrt 5}{2}\right)^{k-7},
\end{equation}
for constants $C_1$ and $C_2$. Given \eqref{Vparticular} and \eqref{Vcomplementary} with initial conditions $V(7)=40, V(8)=61$ (obtained from counting the vertices in Figures \ref{A2-7} and \ref{A2-8}), we find the solution of \eqref{Vequation} given in \eqref{Vformula}.

Now we consider $T(k)$. Recall that $T(k)$ refers to the sum of the degrees of all vertices in the graph corresponding to A2-$k$. Again, we use the fact that the A2-$k$+$2$ tiling is generated by pasting the A2-$k$+$1$ tiling with the A2-$k$ tiling.

In this pasting process, the only vertices of A2-$k$ and A2-$k$+$1$ that undergo changes in their degree are the ones on the intersection line. Assume that $k\geq 7$.  We note that all edges are either horizontal or vertical. Together with Lemma \ref{reflection} this implies all vertices on the intersection line have degree three or four. The first vertex on the intersection line has degree three, while the other $1+D_{k+2}$ vertices in the intersection line each have degree four. Those first and last points on the intersection line of A2-$k+2$ arise from four  points in A2-$k$ and A2-$k+1$ (two in A2-$k$ and two in A2-$k+1$) with three of them of degree two and one of them of degree three. The three vertices of degree two appear in two corners of A2-$k$ and one corner of A2-$k+1$, while the one vertex of degree three appears in the interior of one of the sides of A2-$k+1$, and we know the degree there has to be three due to Lemma \ref{reflection}.

Thus before the pasting, the four corner points had degrees that sum to 9, after the pasting the first and last points on the intersection line have degrees that sum to 7. Thus the pasting process results in a loss of $2=9-7$ degrees in total from the first and last vertices on the intersection line.

We now consider the $D_{k-4}$ middle vertices on the intersection line of A2-$k$+$2$. Before the pasting, each of these $D_{k-4}$ middle vertices arise from a vertex from A2-$k$ and a vertex from A2-$k$+$1$.

From Lemma \ref{reflection}, we can see that after pasting each of the $D_{k-4}$ middle vertices must have degree exactly $4$. Before pasting, the middle vertex corresponds to either two vertices of degree three each in A2-$k$ and A2-$k$+$1$ (this occurs when the vertex is not in a corner of  A2-$k$ or A2-$k$+$1$) or one vertex of degree $2$ and one vertex of degree $4$ in A2-$k$ and A2-$k$+$1$ (this occurs when the vertex appears in a corner of A2-$k$ or A2-$k$+$1$).

Thus the pasting process results in a loss of $2D_{k-4}$ degrees from the middle vertices on the intersection line.

From this we can derive a recursion equation for $T(k)$ when $k\geq 7$:
\begin{equation}
T(k+2)=T(k+1)+T(k)-2-2D_{k-4}
\end{equation}

Using Lemma \ref{DkFormula}, we know this is equivalent to
\begin{equation}\label{Tequation}
T(k+2)=T(k+1)+T(k)-2F(\lfloor (k-4)/2\rfloor+3)- 2F(\lfloor (k-4-1)/2\rfloor+2).
\end{equation}

Again, this is a second order non-homogeneous difference equation. We can check that the following is a particular solution of  \eqref{Tequation}:

\begin{equation}\label{Tparticular}
T_p(k)=\begin{cases}
F\left(\frac{k}{2}\right)+2F\left(\frac{k}{2}+1\right) +4F\left(\frac{k}{2}-1\right)-F\left(\frac{k}{2}-3\right), & \text{ $k$ even}\\
4F\left(\frac{k+1}{2}\right)+2F\left(\frac{k-1}{2}\right), & \text{ $k$ odd}
\end{cases}
\end{equation}
The homogeneous parts of \eqref{Vequation} and \eqref{Tequation} are the same, so the complementary solution of \eqref{Tequation} is just \eqref{Vcomplementary}. Using the initial conditions $T(7)=104$, $T(8)=162$ from counting the degrees of vertices in Figures \ref{A2-7} and \ref{A2-8}, we get the solution of $T(k)$ in \eqref{Tformula}. 

It remains to demonstrate the calculation of average degree for the limiting graph as $k\to\infty$. Notice that by the closed form formula for the Fibonacci number
\begin{equation}\label{FibonacciDefinition}
    F(n) = \dfrac{1}{\sqrt{5}} \cdot \left(\left(\dfrac{1+\sqrt{5}}{2}\right)^n - \left(\frac{1-\sqrt{5}}{2}\right)^n \right),
\end{equation}
 both $V(k)$ and $T(k)$ are linear combinations of powers of $\left( \frac{1\pm \sqrt 5}{2}\right)$. For large $k$, the largest of these terms will be 
\begin{equation}
\left( \frac{31}{2}+\frac{69}{2\sqrt 5}\right)\left( \frac{1+ \sqrt 5}{2}\right)^{k-7}
\end{equation}
for $V(k)$ and
\begin{equation}
\left( 44+\frac{98}{\sqrt 5}\right)\left( \frac{1+ \sqrt 5}{2}\right)^{k-7}
\end{equation}
for $T(k)$. 

We then have
\begin{equation}
\lim_{k\to\infty} \frac{T(k)}{V(k)}=
\frac{ 44+\frac{98}{\sqrt 5}}{ \frac{31}{2}+\frac{69}{2\sqrt 5}} =\frac{2}{11}\left(14+\phi\right),
\end{equation}
as desired. 
\end{proof}

\section{Numerical Calculations}
Through the calculations of equation \eqref{Vformula} and equation \eqref{Tformula}, we can obtain the total number of vertices and sum of the individual degrees of all vertices by increasing the value of $k$, as summarized in Table \ref{tab:my_label}.
\begin{table}[H]
    \centering
    \begin{tabular}{|c|c|c|}
       \hline
        Generation, A2-$k$ & Total Number of Vertices & 2 $\times$ \text{Total Number of Edges} \\
           \hline
           A2-7  & 40 & 104 = 2 $\times$ 52 \\
           \hline
           A2-8  & 61 & 162 = 2 $\times$ 81\\
           \hline
           A2-9  & 95 & 256 = 2 $\times$ 128 \\
           \hline
           A2-10  & 148 & 404 = 2 $\times$ 202 \\
           \hline
           A2-11  & 234 & 644 = 2 $\times$ 322 \\
           \hline
           A2-12  & 370 & 1026 = 2 $\times$ 513 \\
           \hline
           A2-13  & 590 & 1644 = 2 $\times$ 822 \\
           \hline
           A2-14  & 941 & 2634 = 2 $\times$ 1317 \\
           \hline
    \end{tabular}
    \caption{Data for A2-7 through A2-14}
    \label{tab:my_label}
\end{table}
\textbf{Remark:} The difference value is $\left|L - d(k)\right|$, the absolute error between the theoretical limit and the average degree.

We extended $k$ to larger values to verify the correctness of the limit value. The corresponding results are presented in Table \ref{tab:larger_k_values}.
\begin{table}[H]
    \centering
    \begin{tabular}{|c|c|c|c|c|}
        \hline
        A2-$k$ & $V(k)$ & $T(k)$ & Average Degree & Difference \\
        \hline
        15 & 1509 & 4236 & 2.807157057654075 & $3.25 \times 10^{-2}$ \\
        \hline
        20 & 16293 & 46114 & 2.830295218805622 & $9.35 \times 10^{-3}$ \\
        \hline
        30 & 1983866 & 5631810 & 2.838805645139339 & $8.37 \times 10^{-4}$ \\
        \hline
        40 & 243778372 & 692225052 & 2.839567129441655 & $7.54 \times 10^{-5}$ \\
        \hline
        50 & 29980307813 & 85133153674 & 2.839635743735918 & $6.80 \times 10^{-6}$ \\
        \hline
        100 & $8.437515 \times 10^{20}$ & $2.395953 \times 10^{21}$ & 2.839642543368540 & $4.05 \times 10^{-11}$ \\
        \hline
        200 & $6.683110 \times 10^{41}$ & $1.897764 \times 10^{42}$ & 2.839642543409072 & $1.44 \times 10^{-21}$ \\
        \hline
        300 & $5.293497 \times 10^{62}$ & $1.503164 \times 10^{63}$ & 2.839642543409072 & $5.12 \times 10^{-32}$ \\
        \hline
        500 & $3.321014 \times 10^{104}$ & $9.430492 \times 10^{104}$ & 2.839642543409072 & $6.46 \times 10^{-53}$ \\
        \hline
        700 & $2.083525 \times 10^{146}$ & $5.916466 \times 10^{146}$ & 2.839642543409072 & $8.16 \times 10^{-74}$ \\
        \hline
        1000 & $1.035359 \times 10^{209}$ & $2.940048 \times 10^{209}$ & 2.839642543409072 & $3.66 \times 10^{-105}$ \\
        \hline
    \end{tabular}
    \caption{Computed values of $V(k)$, $T(k)$, average degree, and difference for various $k$ values}
    \label{tab:larger_k_values}
\end{table}
The characterization of the exponential convergence rate \( \beta \) requires precise computation of the discrepancy \( |L - \overline{d}(k)| \), which for large \( k \) becomes too small for conventional double-precision arithmetic.  Our analysis therefore employs symbolic computation via the \texttt{SymPy} library. The computational procedure is structured as follows:
\begin{enumerate}
    \item \textbf{Exact Integer Computation:} The sequences for  \( V(k) \) and \( T(k) \) are generated through their respective recurrence relations using exact integer arithmetic.
    \item \textbf{Rational Representation of Average Degree:} The average degree \( \overline{d}(k) \) is computed exactly as a rational number, \( \overline{d}(k) = T(k)/V(k) \), thereby precluding any intermediate floating-point rounding errors.
    \item \textbf{Symbolic Evaluation of the Discrepancy:} The theoretical limit \( L \) is defined symbolically as \( L = (29 + \sqrt{5})/11 \). The absolute discrepancy \( |L - \overline{d}(k)| \) is then computed through an exact symbolic subtraction of the rational number \( \overline{d}(k) \) from \( L \). This algebraic manipulation, performed by \texttt{SymPy}, preserves full precision.
    \item \textbf{High-Precision Numerical Evaluation:} The resultant exact symbolic expression for the discrepancy is subsequently evaluated to a sufficiently high decimal precision. These high-fidelity values are used to compute the logarithmic discrepancy \( \log_{10}|L - \overline{d}(k)| \) plotted in Figure~\ref{Convergence Plot}.
\end{enumerate}

This approach allows us to reliably capture even arbitrarily small differences in asymptotic behavior, overcoming the limitations of standard double-precision arithmetic.

We computed 50 data points for generations $k$ ranging from 10 to 500 in steps of 10. The figure \ref{Convergence Plot} illustrates the logarithmic discrepancy between the average degree $\overline{d}(k)$ and theoretical limit $L$ across iteration counts $k$. Linear regression analysis yields the following statistically significant results:
\begin{itemize}
    \item Slope estimate $\hat{\beta}_1 = -0.104501 \pm 0.000004$ (95\% confidence interval)
    \item Intercept estimate $\hat{\beta}_0 = 0.059770$
    \item Correlation coefficient $r = -1.000000$, indicating perfect negative correlation
    \item Coefficient of determination $R^2 = 1.000000$, demonstrating perfect model fit
    \item Statistical significance $p = 7.235063 \times 10^{-174}$, far below the 0.05 significance threshold
\end{itemize}
These results confirm a perfect linear negative relationship between the logarithmic discrepancy and iteration count. The asymptotic behavior exhibits precise concordance with the theoretical derivation $L = \frac{29 + \sqrt{5}}{11}=\frac{2}{11}\left(14+\phi\right)$, providing numerical validation of the convergence of the average degree.

\begin{figure}[H]
    \centering
    \includegraphics[scale = 0.28]{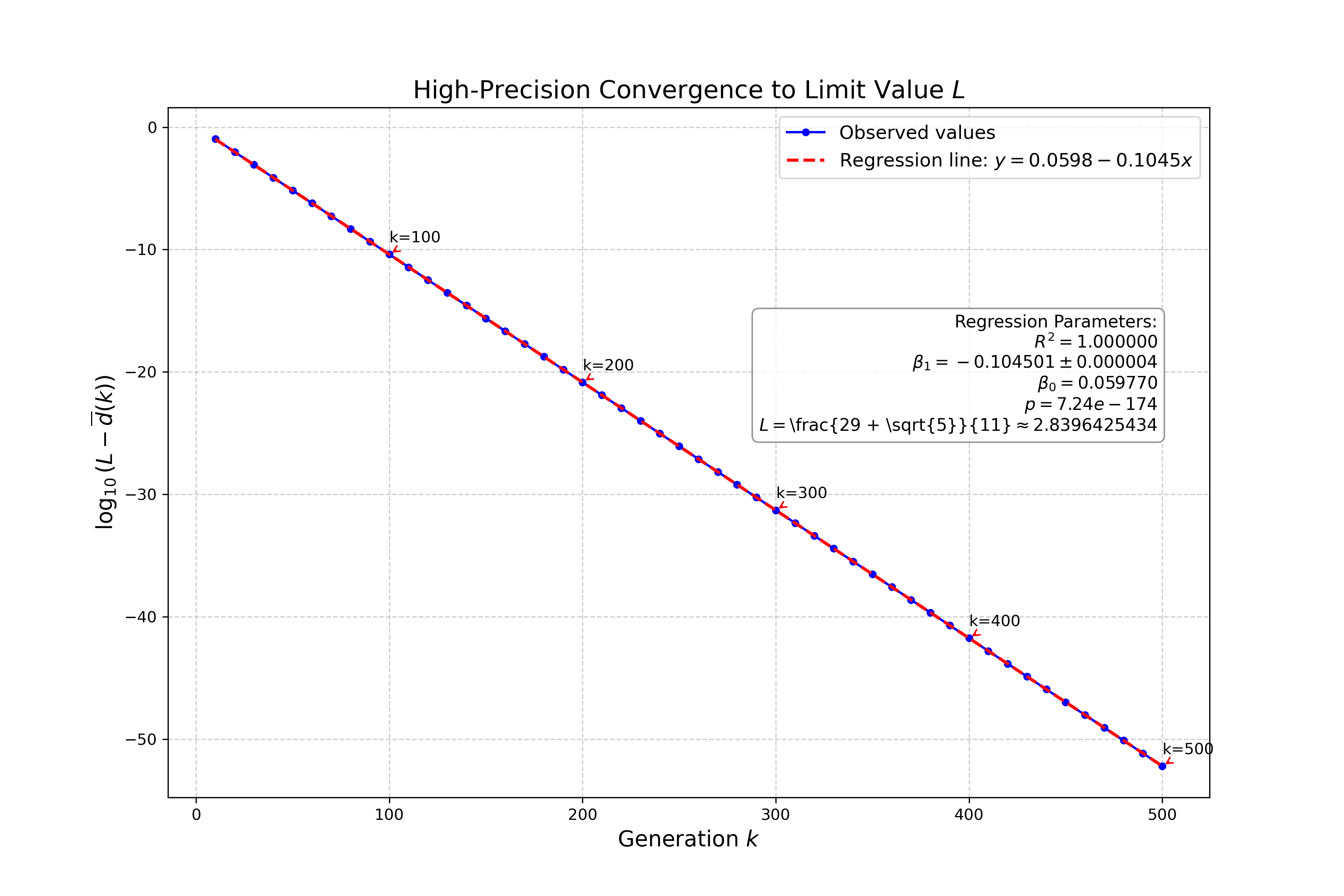}
    \caption{Convergence Plot}
    \label{Convergence Plot}
\end{figure}


\medskip
\noindent MSC2020: 05C10, 05B45

\end{document}